\documentclass[12pt]{article}
\usepackage{amssymb,amsmath,amsthm,url}

\theoremstyle{plain}
\newtheorem{theorem}{Theorem}[section]
\newtheorem{prop}[theorem]{Proposition}
\newtheorem{lemma}[theorem]{Lemma}
\newtheorem{corollary}[theorem]{Corollary}
\newtheorem{thm}{Theorem}
\theoremstyle{definition}

\newtheorem{remark}[theorem]{Remark}
\newtheorem{example}[theorem]{Example}
\newtheorem{question}[theorem]{Question}
\newtheorem{rem}[thm]{Remark}
\newtheorem{que}{Question}

\newtheorem*{thm-simple}{Theorem~\ref{t:simple}}

\newcommand{\F}{\mathcal{F}}
\newcommand{\W}{\mathcal{W}}
\newcommand{\X}{\mathcal{X}}
\newcommand{\PP}{\mathbf{P}}
\newcommand{\Z}{\mathbb{Z}}
\newcommand{\lcm}{\mathop{\mathrm{lcm}}}
\newcommand{\PSL}{\mathop{\mathrm{PSL}}}
\newcommand{\PSU}{\mathop{\mathrm{PSU}}}
\newcommand{\PSp}{\mathop{\mathrm{PSp}}}

\newcommand{\SL}{\mathop{\mathrm{SL}}}
\newcommand{\GL}{\mathop{\mathrm{GL}}}
\newcommand{\Aut}{\mathop{\mathrm{Aut}}}
\newcommand{\Heis}{\mathop{\mathrm{Heis}}}
\newcommand{\leqn}{\trianglelefteq}

\begin{document}
\title{Minimal cover groups}

\author{Peter J. Cameron,\footnote{School of Mathematics and Statistics,
University of St Andrews, St Andrews, Fife KY16 9SS, UK}\ \
David Craven,\footnote{School of Mathematics, University of Birmingham,
Edgbaston, Birmingham B15 2TT, UK}\\ 
Hamid Reza Dorbidi,\footnote{Department of Mathematics, Faculty of Science,
University of Jiroft, Jiroft 78671-61167, Iran}\ \ 
Scott Harper,\footnote{School of Mathematics and Statistics,
University of St Andrews, St Andrews, Fife KY16 9SS, UK}\\
and Benjamin Sambale\footnote{Institut f\"ur Algebra, Zahlentheorie und Diskrete
Mathematik, Leibniz Universit\"at Hannover, 30167 Hannover, Germany}}

\date{}

\maketitle

\begin{abstract}
Let $\F$ be a set of finite groups. A finite group $G$ is called an \emph{$\F$-cover} if every group in $\F$ is isomorphic to a subgroup of $G$. An $\F$-cover is called \emph{minimal} if no proper subgroup of $G$ is an $\F$-cover, and \emph{minimum} if its order is smallest among all $\F$-covers. We prove several results about minimal and minimum $\F$-covers: for example, every minimal cover of a set of $p$-groups (for $p$ prime) is a $p$-group (and there may be finitely or infinitely many, for a given set); every minimal cover of a set of perfect groups is perfect; and a minimum cover of a set of two nonabelian simple groups is either their direct product or simple. Our major theorem determines whether $\{\Z_q,\Z_r\}$ has finitely many minimal covers, where $q$ and $r$ are distinct primes. Motivated by this, we say that $n$ is a \emph{Cauchy number} if there are only finitely many groups which are minimal (under inclusion) with respect to having order divisible by $n$, and we determine all such numbers. This extends Cauchy's theorem. We also define a dual concept where subgroups are replaced by quotients, and we pose a number of problems.
\end{abstract}

\clearpage
\section{Introduction}

Cayley's celebrated theorem asserts that every group of order $n$ is isomorphic to a subgroup of the symmetric group $S_n$. This motivates the following problem: given a finite set $\F$ of finite groups, find a group $G$ such that every group in $\F$ is isomorphic to a subgroup of $G$. We call a group $G$ with this property an \emph{$\F$-cover}. When $\F$ consists of all the groups of order $n$ up to isomorphism, which is the case occurring in Cayley's theorem, we refer to an $\F$-cover as an $n$-cover. 

It is natural to ask what the smallest $\F$-cover is, and to make this precise we introduce the following definitions. For a finite set $\F$ of finite groups, we say that an $\F$-cover is
\begin{itemize}
\setlength\itemsep{0pt}
\item \emph{minimal} if no proper subgroup of $G$ is an $\F$-cover;
\item \emph{co-minimal} if no proper quotient of $G$ is an $\F$-cover;
\item \emph{strongly minimal} if it is both minimal and co-minimal;
\item \emph{minimum} if no $\F$-cover has smaller order.
\end{itemize}
Note that a minimum cover is strongly minimal. 

In this paper, we will mainly focus on minimal and minimum covers, guided by the following two questions.

\begin{que} \label{q:finite}
For which finite sets $\F$ of finite groups are there only finitely many minimal $\F$-covers up to isomorphism?
\end{que}

\begin{que} \label{q:property}
Given a group theoretic property $\PP$, is it is true that if every group in $\F$ has property $\PP$, then every (or perhaps some) minimum (or minimal) $\F$-cover has property $\PP$?
\end{que}

Our first theorem, which we prove in Section~\ref{s:twocyclic}, fully answers Question~\ref{q:finite} in an important special case.

\begin{thm} \label{t:twocyclic}
Let $q < r$ be primes. The set $\{\Z_q,\Z_r\}$ has only finitely many minimal covers if and only if $q = 2$ and $r$ is a Fermat prime, in which case there are exactly three: $\Z_{2r}$, $D_{2r}$ and $\Z_2^{2a} : \Z_r$ where $r = 2^a+1$.
\end{thm}

Cauchy's theorem asserts that if a prime $p$ divides the order of $G$, then $G$ has an element of order $p$, or, equivalently, has a subgroup isomorphic to $\Z_p$. More generally, Sylow's theorem asserts that if a prime power $p^a$ divides the order of $G$, then $G$ has a subgroup isomorphic to one of the (finitely many) groups of order $p^a$. These theorems are sharp, in the sense that if every group of order divisible by $n > 1$ necessarily has an element of order $n$, then $n$ is prime, and if it necessarily has a subgroup of order $n$, then $n$ is a prime power (see \cite{McCarthy70}). Nevertheless, Theorem~\ref{t:twocyclic} shows that for $n=2p$ where $p$ is a Fermat prime, there is a set $\W$ of three groups of order divisible by $n$ such that if $n$ divides the order of $G$, then $G$ has a subgroup isomorphic to a group in $\W$. More generally, we say that $n$ is a \emph{Cauchy number} if there is a finite set $\W$ of groups of order divisible by $n$ such that if $n$ divides the order of $G$, then $G$ has a subgroup isomorphic to a group in $\W$. Cauchy numbers form the focus of Section~\ref{s:cauchy}, where we completely determine them. 

\begin{thm} \label{t:cauchy}
Let $n$ be prime. Then $n$ is a Cauchy number if and only if one of the following holds:
\begin{enumerate}
\setlength\itemsep{0pt}
\item $n$ is a prime power;
\item $n=6$;
\item $n=2p^a$, where $p > 3$ is a Fermat prime and $a \geq 1$.
\end{enumerate}
\end{thm}

\begin{rem} \label{r:six}
Theorem~\ref{t:twocyclic} guarantees that every group of order divisible by $6$ contains a subgroup isomorphic to $\Z_6$, $S_3$ or $A_4$. While our proof of Theorem~\ref{t:twocyclic} uses the Classification of Finite Simple Groups, as we explain in Remark~\ref{r:six-cfsg}, this important special case can be proved without it.
\end{rem}

\begin{rem} \label{r:cauchy}
We say that $G$ is an \emph{$n$-witness} if $n$ divides the order of $G$ but $n$ does not divide the order of any proper subgroup of $G$. Then $n$ is a Cauchy number if and only if there are finitely many $n$-witnesses up to isomorphism. If $n = p^a$, then the $n$-witnesses are the groups of order $n$, and if $n = 2p$, for a Fermat prime $p$, then the $n$-witnesses are $\Z_{2p}$, $D_{2p}$, $\Z_2^{2a} : \Z_p$ as in Theorem~\ref{t:twocyclic}. Now consider $n = 2p^a$ for a Fermat prime $p > 3$ and $a > 1$. Proposition~\ref{p:witness} shows that any $n$-witness has order dividing $2^dp^a$ where $d$ is the maximum degree of a (not necessarily faithful) irreducible representation of group of order $p^a$ over $\mathbb{F}_2$, which is sharp since if $P$ is a group of order $p^a$ and $V$ is an irreducible $\mathbb{F}_2P$-module, then $V{:}P$ is a $n$-witness.
\end{rem}

In the process of proving Theorem~\ref{t:cauchy} we establish the following, which could be of independent interest (see Proposition~\ref{p:35}(a)).

\begin{thm} \label{t:60}
Every finite simple group with order divisible by $60$ has a subgroup isomorphic to $A_5$. 
%Conversely, if $S$ is a nonabelian finite simple group and $S \not\cong A_5$, then there is a simple group with order divisible by $|S|$ that has no subgroup isomorphic to $S$.
\end{thm}

Turning to Question~\ref{q:property}, in Section~\ref{s:simple}, we begin with simple groups.

\begin{thm} \label{t:simple}
Let $M$ and $N$ be nonabelian finite simple groups. If $G$ is a minimum cover of $\{ M, N \}$ then either $G = M \times N$ or $G$ is simple and $|G| \leq |M| \cdot |N|$.
\end{thm}

\begin{rem} \label{r:simple}
Both possibilities in Theorem~\ref{t:simple} arise. In fact, they can arise simultaneously. By Corollary~\ref{c:simple_product}, $A_{12}$ and $A_7 \times {\rm M}_{12}$ are both minimum covers of $\{A_7, {\rm M}_{12} \}$. In particular, $\{M,N\}$ can have nonisomorphic minimum covers, but we do not know whether $\{M,N\}$ can have nonisomorphic minimum \emph{simple} covers. Section~\ref{s:simple} gives further results and questions in this direction.
\end{rem}

Section~\ref{s:fp} considers $p$-groups. Every minimal cover for a set of $p$-groups is a $p$-group, and we focus on $p^n$-covers, i.e. $\F$-covers where $\F$ is the set of all groups of order $p^n$. We prove that there are only finitely many minimal $p^2$-covers, but infinitely many minimal $2^3$-covers (see Theorem~\ref{t:inf8}), and we determine (rather weak) upper and lower bounds for the order of a minimum $p^n$-cover: the upper bound is $p^{(p^n-1)/(p-1)}$, and the lower bound $p^{(c+o(1))n^2}$ with $c=2/27$
(see Theorem~\ref{t:227}).

Many other properties are inherited by minimal or minimum covers. For instance, if every group in $\F$ is perfect, then every minimal $\F$-cover is perfect (see Example~\ref{ex:residual}(a)), and if every group in $\F$ is nilpotent, then every minimum $\F$-cover is nilpotent (see Theorem~\ref{t:can}). However, for other properties, this fails. For instance, if every group in $\F$ is soluble, then there need not exist a soluble $\F$-cover (see Example~\ref{ex:a5}). 

We leave open the question of whether if $\F$ is a finite set of finite abelian groups, then there necessarily exists an abelian minimum $\F$-cover. However, Section~\ref{s:can} does present a number of results in this direction. In particular, the main result of that section is an algorithm to find the smallest abelian group containing a given set of finite abelian groups. Since abelian groups are direct products of their Sylow subgroups, it suffices to solve the problem for abelian $p$-groups. Our algorithm has the following consequence (see Corollary~\ref{c:abelian}).

\begin{thm} \label{t:abelian}
The smallest abelian group which embeds all abelian groups of order $p^n$ is unique up to isomorphism and has order $p^{f(n)}$, where
\[
f(n)=\sum_{k=1}^n\lfloor n/k\rfloor.
\]
\end{thm}

This function $f$ was studied by Dirichlet who found its asymptotic behaviour. It is sequence A006218 in the On-Line Encyclopedia of Integer Sequences~\cite{oeis}, where several further areas where it arises are given.

We conclude the paper with a dual version of the ideas introduced in the paper and number of further directions and open problems.

\paragraph{Notation.} 
Our notation is fairly standard, but note that we write $\Z_n$ for the cyclic group of order~$n$ and $D_{2n}$ for the dihedral group of order $2n$. For sporadic simple groups we follow the \emph{$\mathbb{ATLAS}$ of Finite Groups} \cite{ATLAS}. We write $n_p$ for the $p$-part of $n$, and $o_p(n)$ for the order of $n$ modulo $p$.
%For undefined properties of finite groups and more detailed information, we refer to Hall~\cite{hall}, Robinson~\cite{rob} or Suzuki~\cite{suzuki}.

\paragraph{Acknowledgements.} 
We are grateful to Jon Awbrey and Michael Kinyon for helpful comments. 
The fourth author is an EPSRC Postdoctoral Fellow (EP/X011879/1). 
In order to meet institutional and research funder open access requirements, any accepted manuscript arising shall be open access under a Creative Commons Attribution (CC BY) reuse licence with zero embargo.

\section{General results} \label{s:prelims}

We begin with a simple result to illustrate the central concepts of the paper, showing in particular that any finite set of finite groups has a cover.

\begin{prop} \label{p:order}
Let $\F$ be a finite set of finite groups, and let $G$ be an $\F$-cover. Then the following hold:
\begin{enumerate}
\item We have that $\lcm\{|F|:F\in\F\}$ divides $|G|$.
\item If $G$ is a minimum $\F$-cover, then $|G|\le\prod_{F\in\F}|F|$.
\item If the orders of the groups in $\F$ are pairwise coprime, then
equality holds in (b).
\end{enumerate}
\end{prop}

\begin{proof}
(a) is immediate from Lagrange's Theorem; (b) follows from the fact that the
direct product of the groups in $\F$ is a cover; and (c) is immediate from (a)
and (b). 
\end{proof}

Note that, if the orders of the groups in $\F$ are powers of distinct primes,
then a minimum $\F$-cover is a group which has these as its Sylow subgroups.

\begin{example} 
A minimum cover of $\{\Z_3$, $(\Z_2)^2,
\Z_5\}$ has order $60$; any group of order $60$ having these three
groups as its Sylow subgroups is an example. This holds for seven of the
$13$ groups of order 60, including $S_3\times D_{10}$, $A_4\times\Z_5$,
and $A_5$.
\end{example}

Proposition~\ref{p:order} raises the following natural problem.

\begin{question}
Characterise sets $\F$ of finite groups for which the order of a minimum $\F$-cover is the least common multiple of the orders of the groups in $\F$.
\end{question}

Proposition~\ref{p:order} shows that the order of a minimum $\F$-cover is bounded by a function of $\F$ (namely, the product of the orders of the groups in $\F$). In particular, there are only finitely many minimum $\F$-covers. However, there may be infinitely many minimal $\F$-covers, which motivates Question~\ref{q:finite} in the introduction. In this context, the following result will be useful.

\begin{prop}\label{piFcovers}
Let $\pi$ be a finite set of primes, and $\F$ a finite set of $\pi$-groups.
If there exists a minimal $\F$-cover which is not a $\pi$-group, then
there exist infinitely many minimal $\F$-covers.
\end{prop}
\begin{proof}
Let $G$ be a minimal $\F$-cover whose order is divisible by a prime
$p\notin\pi$. By a theorem of Gasch\"utz (see \cite[Chapter B, Theorem 11.8]{dh}), there exists a Frattini extension $H$ of $G$, i.e. $H$ has an elementary abelian normal $p$-subgroup $E\le\Phi(H)$ such that $H/E\cong G$. Let $M<H$ be a maximal subgroup. Then $E\le\Phi(H)\le M$. Suppose by way of contradiction that $M$ is an $\F$-cover. For every $F\in\F$, we may assume that $F\le M$. Since $F$ is a $\pi$-group, it follows that $F\cap E=1$ and 
\[F\cong FE/E\le M/E<H/E\cong G.\] 
But this implies that $M/E$ is an $\F$-cover. This contradiction shows that $H$ is a minimal $\F$-cover. Now we can repeat this process with $H$ instead of $G$. This yields an infinite series of minimal $\F$-covers.
\end{proof}

Let us now turn to Question~\ref{q:property} on whether minimal covers inherit group-theoretic properties. We begin with some simple observations.

\begin{prop} \label{p:pgroup}
Let $p$ be a prime, and let $\F$ be a finite set of finite $p$-groups. Then every minimal $\F$-cover is a $p$-group. 
\end{prop}

\begin{proof}
Let $G$ be a minimal $\F$-cover, and $P$ a Sylow $p$-subgroup of $G$. Then every group in $\F$ is embedded in $G$, and so conjugate to a subgroup of $P$. Hence $P$ is an $\F$-cover. By minimality, $P=G$. 
\end{proof}

\begin{prop}\label{hrd4}
Let $\mathcal{A}$ be a class of groups which is closed under the taking of subgroups and direct products. If $\mathcal{F}$ is a subset of $\mathcal{A}$ then there exists an $\mathcal{A}$-group which is a minimal $\mathcal{F}$-cover.
\end{prop}

\begin{proof}
Let $G=\Pi_{H\in\F} H$. Then $G$ is an $\mathcal{A}$-group which is also an $\mathcal{F}$-cover. So $G$ contains a minimal $\mathcal{F}$-cover which is an $\mathcal{A}$-group. 
\end{proof}

It is not clear when we can obtain a minimum $\F$-cover in this way. We will see in Section~\ref{s:can} that this is the case for nilpotent groups, but we have been unable to decide the apparently easier case of abelian groups.

The following two theorems give further results in this direction.

\begin{prop} \label{p:residual}
Suppose that $\X$ is a subgroup-closed class of finite groups. Let 
$\F$ be a finite set of finite groups, none of which has a non-trivial
$\X$-group as a quotient, and let $G$ be a minimal $\F$-cover. Then $G$ has
no non-trivial $\X$-group as a quotient.
\end{prop}

\begin{proof} 
Suppose that $G/N\in\X$. Then for any group $H\in\F$ with $H\le G$, we have $H/H\cap N\cong HN/N\leq G/N\in\X$. So $H/H\cap N\in\X$ which implies $H\subseteq N$. By minimality of $G$, we have $N=G$.
\end{proof}

\begin{example} \label{ex:residual}
Let us record some applications of Proposition~\ref{p:residual}.
\begin{enumerate}
\item Let $\X$ be the class of finite abelian groups. The condition
that $G$ has no non-trivial homomorphism to an $\X$-group means that
$G$ is perfect. So we deduce that, if every group in $\F$ is perfect,
then any minimal $\F$-cover is perfect.
\item Let $\X$ be the class of finite soluble groups. The condition
that $G$ has no non-trivial homomorphism to an $\X$-group means that
$G$ is equal to its soluble residual. So, if every group in $\F$ is
equal to its soluble residual, then the same is true of any minimal
$\F$-cover.
\end{enumerate}
\end{example}

There is a dual result, as follows.

\begin{prop} \label{p:radical}
Suppose that $\X$ is a subgroup-closed class of finite groups. Let
$\F$ be a finite set of finite groups, and suppose that no group in
$\F$ has a non-trivial normal $\X$-subgroup. Let $G$ be
a co-minimal cover of $\F$. Then $G$ has no non-trivial normal
$\X$-subgroup.
\end{prop}

\begin{proof}
Let $G$ be a co-minimal cover of $\F$, and suppose that $G$ has a normal $\X$-subgroup $N$. For each group $H\in\F$,
we have $H\cap N\in\X$. So $H\cap N=\{1\}$. Hence $H\cong H/H\cap N\cong HN/N\leq G/N$.
By co-minimality, $G/N\cong G$, so $N=\{1\}$.
\end{proof}

\begin{example} \label{ex:radical}
We now give some applications of Proposition~\ref{p:radical}.
\begin{enumerate}
\item Let $\X$ be the class of finite abelian groups. If no group in
$\F$ has a non-trivial abelian normal subgroup, then a co-minimal
$\F$-cover has no non-trivial abelian normal subgroup.
\item Let $\X$ be the class of finite soluble groups. The condition
``no non-trivial normal $\X$-subgroup'' means that the soluble
radical is trivial. So, if all groups in $\F$ have trivial soluble
radical, then the same is true of a co-minimal cover.
\end{enumerate}
\end{example}

We conclude with some remarks on $n$-covers, which, recall, are $\F$-covers where $\F$ consists of the groups of order $n$.

\begin{prop}\label{hrd40}
Let $n=p_1^{\alpha_1}\cdots p_k^{\alpha_k}$ and $G$ be a minimal $n$-cover. If $P_i$ is a Sylow $p_i$-subgroup of $G$ then $P_i$ is a  $p_i^{\alpha_i}$-cover.
\end{prop}

\begin{proof}
Let $P$ be a group of order $p_i^{\alpha_i}$. Then $H=P\times\Pi_{j\neq i} \mathbb{Z}_{p_j^{\alpha_j}}$ is a group of order $n$. So $H$ is isomorphic to a subgroup of $G$. Hence $P$ is isomorphic to a subgroup of $P_i$.
\end{proof}

\begin{question}\label{hrd103}
In the above situation, when is $P_i$ a minimal $p_i^{\alpha_i}$-cover group?
\end{question}

\begin{remark}\label{hrd5}
If $\mathcal{F}$ is a set of soluble $\pi$-groups, then every soluble minimal $\F$-cover is a $\pi$-group, by Hall's theorem. Moreover, if $\mathcal{F}$ is a set of $\pi$-groups, then every minimum $\F$-cover $G$ satisfies $O_{\pi'}(G)=1$.
\end{remark}

\section{Two cyclic groups of prime order} \label{s:twocyclic}

In this section, we examine minimal covers for $\{\Z_q,\Z_r\}$, where $q$ and $r$ are distinct primes. In particular, we will prove Theorem~\ref{t:twocyclic}, which determines when there are finitely many such covers.

Our first result handles the generic case. 

\begin{prop} \label{p:qr}
Let $q$ and $r$ be distinct odd primes with at least one strictly greater than $5$. Then there are infinitely many minimal $\{\Z_q,\Z_r\}$-covers.
\end{prop}

\begin{proof}
Without loss of generality we may assume that $q<r$, so $r\geq 7$. By Dirichlet's theorem, there are infinitely many primes $p$ satisfying
$p\equiv1$ (mod~$q$) and $p\equiv-1$ (mod~$r$). We claim that, for any such
prime $p$, the group $G=\PSL_2(p)$ is a minimal $\F$-cover, where
$\F=\{\Z_q,\Z_r\}$. To see this, we consult the list of subgroups of $\PSL_2(p)$ given in Suzuki~\cite[Theorem~3.6.25]{suzuki}. The subgroup $A_5$ has order not divisible by $r\geq 7$, and $r\nmid p(p-1)$, so the only subgroup of order divisible by $r$ is dihedral of order $p+1$. But this cannot be divisible by $q\mid (p-1)$, so no maximal subgroup of $G$ has order divisible by $qr$.
\end{proof}

\begin{remark} \label{r:qr}
Since $\PSL_2(p)$ is simple, we see that the minimal covers exhibited in the proof of Proposition~\ref{p:qr} are strongly minimal.
\end{remark}

To handle the remaining pairs of primes, we first prove a general result, which reduces our study to insoluble groups.

\begin{prop} \label{p:qr2}
Let $G$ be a minimal $\{\Z_q,\Z_r\}$-cover. One of the following holds:
\begin{enumerate}
\item\label{twocyclic:solcase} $G$ is soluble, has an elementary abelian normal subgroup $N$, and $G/N$
is cyclic. Either $N$ is a $q$-group and $|G/N|=r$, or $N$ is an $r$-group and
$|G/N|=q$.  There are three such groups, namely $\Z_{qr}$, $\Z_q^a:\Z_r$ and $\Z_r^b:\Z_q$, where $a$ is the multiplicative order of $q$ modulo $r$ and $b$ is the multiplicative order of $r$ modulo $q$.
\item $G$ is not soluble and has a unique maximal normal subgroup $N$. (This
includes the case where $G$ is simple.) The group $N$ is equal to $\Phi(G)$
(and in particular $N$ is nilpotent), and the quotient $G/N$ is a nonabelian
simple group that is also a minimal $\{\Z_q,\Z_r\}$-cover. Moreover, $|N|$ is
coprime to $qr$, but if $|N|\ne1$, it is not coprime to $|G/N|$.
\end{enumerate}
\end{prop}

\begin{proof}
Suppose that $G$ is soluble, so by Hall's theorem there exists a
Hall $\{q,r\}$-subgroup. Thus by minimality $G$ is a $\{q,r\}$-group. Let $N$
be a minimal normal subgroup of $G$, which without loss of generality we may
assume to be an elementary abelian $q$-group. If $x$ has order $r$ in $G$ then
$\langle N,x\rangle$ is a $\{\Z_q,\Z_r\}$-cover, so by minimality
$G=\langle N,x\rangle$.

To complete the proof of \ref{twocyclic:solcase} we need to determine $a$ and $b$. By minimality $N$ is an irreducible module for $\Z_r$ over $\mathbb F_q$. If it is the trivial module we obtain $\Z_{qr}$, so we may assume that it is non-trivial. Since $\Aut(\Z_r)$ acts transitively on the set of isomorphism classes of
non-trivial modules over an algebraically closed field (of characteristic different from $r$), we see that $\Aut(\Z_r)$ acts transitively on all non-trivial irreducible modules over any field, in particular $\mathbb F_q$, and this implies there is a unique group $(\Z_q)^a: \Z_r$ up to isomorphism. So it suffices to determine $a$: but if $d$ is the multiplicative order of $q$ modulo $r$ then $\Z_r$ is a subgroup of $\GL_d(q)$ (since the cyclotomic polynomial $\Phi_d(q)$ divides $|\GL_d(q)|$, and $r\mid \Phi_d(q)$). It is also not a subgroup of any smaller linear group, so $d=a$. This proves \ref{twocyclic:solcase}.

Thus we may assume that $G$ is not soluble. Let $N$ be any maximal normal
subgroup. By minimality of $G$, at least one of $q$ and $r$ does not divide
$|N|$. Suppose exactly one does, so $q\mid|N|$ without loss of generality. If
$x\in G$ has order $r$, then $x$ acts on the
Sylow $q$-subgroups of $N$. The number of these is prime to $r$, so $x$
normalizes one, say $P$. Thus $\langle P,x\rangle$ is a soluble
$\{q,r\}$-group contained in $G$, contradicting minimality.

Thus each of $q$ and $r$ does not divide $|N|$. In particular, $G/N$ must be
nonabelian simple. If $G/N$ is not a minimal $\{\Z_q,\Z_r\}$-cover then we
may choose a proper subgroup $M/N$ of $G/N$ with order divisible by $qr$,
whence its preimage $M$ also has order divisible by $qr$, contradicting
minimality of $G$.

To prove uniqueness of $N$, let $M$ be any other maximal normal subgroup
of $G$, and note the above discussion also holds for $M$. Clearly $MN=G$, so
$G/N\cong M/M\cap N$. But then $|M/M\cap N|$ and so $|M|$ are divisible by $qr$,
and this contradicts the minimality of $G$.

If $H$ is any maximal subgroup of $G$, then $NH<G$, since otherwise
$H/(H\cap N)\cong G/N$ and so $H$ is a cover. But this implies that $N\le H$.
So $N\le\Phi(G)$, and we must have equality as $G/N$ is simple. Finally, if 
$\gcd(|N|,|G/N|)=1$ then the Schur--Zassenhaus theorem implies
that $N$ has a complement, which is a cover. 
\end{proof}

\begin{corollary} \label{c:qr2}
If there are no nonabelian simple minimal $\{\Z_q,\Z_r\}$-covers, then all minimal $\{\Z_q,\Z_r\}$-covers are soluble.
\end{corollary}

We next consider $\{q,r\}=\{3,5\}$. Here there is a unique simple minimal cover, but infinitely many insoluble covers.

\begin{prop} \label{p:35} \quad
\begin{enumerate}
\item The only simple minimal $\{\Z_3,\Z_5\}$-cover is the alternating group $A_5$.
\item There are infinitely many insoluble $\{\Z_3,\Z_5\}$-covers.
\end{enumerate}
\end{prop}

\begin{proof} (a) Let $G$ be a simple group, and we use the classification of finite simple groups. First, it is easy that $A_5$ actually \emph{is} a minimal $\{\Z_3,\Z_5\}$-cover so we assume that $G$ is not $A_5$. We will show that $G$ always possesses a proper subgroup of order divisible by $15$, or $|G|$ is not divisible by $15$.

If $G$ is an alternating group then clearly $G$ contains $A_5$ and we are done. If $G$ is a sporadic group then all maximal subgroups of $G$ are known and we may check the \emph{$\mathbb{ATLAS}$ of Finite Groups}
\cite{ATLAS} (except for the Monster, where all maximal subgroups have only recently been identified~\cite{dlp}; however, $2\cdot B$ is a subgroup of $M$, and we are done).

Thus let $G$ be a group of Lie type. Note that $G$ cannot be a Suzuki group ($3$ does not divide their orders) or a small Ree group ($5$ does not divide their orders). For the large Ree groups, they all contain the Tits group, which contains a maximal subgroup $A_6.2^2$ (see \cite[p.~74]{ATLAS} for example).

Thus we may assume that $G$ is a simple Chevalley or Steinberg group, in characteristic $p$. Let $G=G(p^a)$ be the group; in all cases there is a Levi subgroup of $G$ that is either $\PSL_2(p^a)$ or $\SL_2(p^a)$. If $p^a\equiv 0,\pm 1$ (mod~$5$) then $15$ divides the order of $\PSL_2(p^a)$, and we are done, unless
$G=\PSL(2,p^a)$. But, if $p=5$, then $G$ contains $\PSL_2(5)$, whereas if
$p\equiv\pm1$ (mod~$5$) then $G$ contains $\PSL(2,p)$, which contains $A_5$
because of the congruence condition.

So we may suppose that this is not the case. In particular, $p\equiv\pm2$ (mod~$5$) and $a$ is odd.
These conditions imply that $p$ has order $4$ modulo $5$. A standard fact about cyclotomic polynomials is that if $5\mid \Phi_d(p^a)$ then $d$ is a power of $5$ times $4$. Thus if $5\mid |G|$, then $\Phi_4(p^a)=p^{2a}+1$ (or $\Phi_{20},\Phi_{100}$, etc.) divides $|G(p^a)|$, so this excludes $\PSL_2$, $\PSL_3$, $\PSU_3$, ${}^3D_4$ and $G_2$ from consideration. The groups $\PSp_4(p^a)$ contain groups of the form $\PSL_2(p^{2a})$, so these cannot be minimal covers. Since both $\PSL_4(p^a)$ and $\PSU_4(p^a)$ contain $\PSp_4(p^a)$, these cannot be minimal covers either.

The remaining groups we have not considered are symplectic and orthogonal groups (which all contain $\PSp_4(p^a)=\Omega_5(p^a)$) or exceptional groups $F_4(p^a)$, $E_6(p^a)$, $E_7(p^a)$ and $E_8(p^a)$ (each of which contains the previous one and the first one contains $\PSp_4(p^a)$) and ${}^2E_6(p^a)$ (which contains $F_4(p^a)$).

This completes the list of groups of Lie type, so $A_5$ is the only simple minimal $\{\Z_3,\Z_5\}$-cover, as claimed.

\medskip

(b) This follows from (a) and Proposition~\ref{piFcovers} applies to $\pi=\{3,5\}$.
\end{proof}

Note that Theorem~\ref{t:60} is an immediate consequence of Proposition~\ref{p:35}(a).

\medspace

We now consider the case where $q=2$. Let us outline our strategy in this case. Suppose that $G$ is a simple minimal
$\{\Z_2,\Z_r\}$-cover. If $M$ is any maximal subgroup of $G$ then $M$ is not
divisible by $2r$, so any maximal subgroup of order divisible by $r$ has
odd order. Of course, if $R$ denotes a Sylow $r$-subgroup of $G$ then $N_G(R)$
must be contained in a maximal subgroup of $G$, which therefore must have odd
order. Fortunately, there are very few odd-order maximal subgroups of simple
groups, and they are enumerated in \cite[Table 2]{ls1991}. The following lemma, which makes use of \cite{ls1991}, is stated in greater generality than required here so that it can be used in the next section also.

\begin{lemma} \label{l:liebeck-saxl}
Let $A$ be a finite almost simple group with socle $S$. Let $p$ be a Fermat prime. Let $M$ be a maximal subgroup of $A$ such that $N_S(P) \leq M$ where $P$ is a Sylow $p$-subgroup of $S$. Then $M$ has even order, or $p=3$ and $S=\PSL_2(3^a)$ for an odd integer $a > 1$.
\end{lemma}

\begin{proof}
Suppose that $M$ has odd order. According to \cite[Table~2]{ls1991}, Table~\ref{tab:odd_order} gives the only possibilities for $S$ and $M \cap S$ where $S$ is a finite nonabelian simple group and $M$ is an odd-order maximal subgroup of an almost simple group with socle $S$. (Not all examples in the table necessarily arise. In particular, at the time of writing \cite{ls1991}, it was not known if the subgroups in the final two rows were maximal subgroups of the Monster. They have since been proved not to be~\cite{hw1,hw2}.) In the first row, we require $p = n$ but we must have $n \geq 5$ and $n \equiv 3 \pmod{4}$, which is impossible since $p$ is a Fermat prime. In the second row, we require $q = p^a$ but we must have $q \geq 4$ and $q \equiv 3 \pmod{4}$, which, since $p$ is a Fermat prime, forces $p = 3$ and $a > 1$ to be odd; this gives the exception in the statement. For a contradiction, suppose that one of the remaining rows holds. The restriction that $p$ is a Fermat prime means that we must have one of the following:
\begin{itemize}
\setlength\itemsep{0pt}
\item[1.] $S = \mathrm{PSL}_d(q)$ and $M = \left(\frac{1}{(d,q-1)} \frac{q^d-1}{q-1} \right) {:} d$ for an odd prime $d$
\item[2.] $S = \mathrm{PSU}_d(q)$ and $M = \left(\frac{1}{(d,q+1)} \frac{q^d+1}{q+1} \right) {:} d$ for an odd prime $d$.
\end{itemize}
Since $M \cap S$ contains $N_S(P)$, we may apply the theory of Sylow tori (see \cite[Theorem~25.19]{MalleTesterman11} for instance). In Case~1, this implies that $p$ is a primitive prime divisor of $q^d-1$, so $o_p(q) = d$, while in Case~2, this implies that $p$ is a primitive prime divisor of $q^{2d}-1$, so $o_p(q) = 2d$. In both cases, $d$ divides $p-1$, which is impossible since $d$ is odd and $p$ is a Fermat prime. 
\end{proof}

\begin{table}[htbp]
\[
\begin{array}{lll}
\hline
S                 & M \cap S                                                & \text{conditions}                       \\
\hline  
A_n               & n{:}\left(\frac{n-1}{2}\right)                          & \text{$n \equiv 3 \pmod{4}$ prime}       \\
\mathrm{PSL}_2(q) & q{:}\left(\frac{q-1}{2}\right)                          & \text{$q \equiv 3 \pmod{4}$ prime power} \\
\mathrm{PSL}_d(q) & \left(\frac{1}{(d,q-1)} \frac{q^d-1}{q-1} \right) {:} d & \text{$d$ odd prime}                    \\
\mathrm{PSU}_d(q) & \left(\frac{1}{(d,q+1)} \frac{q^d+1}{q+1} \right) {:} d & \text{$d$ odd prime}                    \\
\mathrm{M}_{23}   & 23{:}11 & \\
\mathrm{Th}       & 31{:}15 & \\
\mathbb{B}        & 47{:}23 & \\
\mathbb{M}        & 59{:}29 & \\
\mathbb{M}        & 71{:}35 & \\
\hline
\end{array}
\]
\caption{Possibilities for an odd-order maximal subgroup $M$ of an almost simple group with socle $S$} \label{tab:odd_order}
\end{table}

\begin{prop} \label{p:2p}
Let $r$ be an odd prime. If $r$ is a Fermat prime then there are no simple minimal $\{\Z_2,\Z_r\}$-covers, and if $r$ is not a Fermat prime then there are infinitely many simple minimal $\{\Z_2,\Z_r\}$-covers.
\end{prop}

%\begin{proof} The odd-order maximal subgroups of simple groups are given in \cite[Table 2]{ls1991}. If $M$ is a maximal odd-order subgroup of a simple group, then $M$ is one of the following:
%\begin{itemize}
%\item $\Z_p:\Z_{(p-1)/2}\leq A_p$, where $p$ is a prime and $p\equiv 3$ (mod $4$);
%\item $\Z_{p^a}: \Z_{(p^a-1)/2} \leq \PSL_2(p^a)$ if $p^a\equiv 3$ (mod $4$);
%\item $\Z_{\Phi_d(p^a)/\gcd(p^a-\epsilon,d)}:d\leq \PSL^\epsilon_d(p^a)$, where $\epsilon=\pm 1$, $d$ an odd prime, $(d,p^a)\neq (3,3),(5,2)$,
%\item $23:11\leq M_{23}$;
%\item $31:15\leq Th$;
%\item $47:23\leq B$;
%\item $59:29\leq M$ and $71:35\leq M$. 
%\end{itemize}
%\textbf{Note}: This list is complete, but not every subgroup on this list is guaranteed to be maximal.

\begin{proof}
First assume that $r$ is a Fermat prime. For a contradiction, suppose that $S$ is a simple minimal $\{\Z_2,\Z_r\}$-cover. As indicated above, let $R$ be a Sylow $r$-subgroup of $S$ and let $M$ be a maximal subgroup of $S$ containing $N_S(R)$. Since $r$ divides $|M|$ we must have that $|M|$ is odd. By Lemma~\ref{l:liebeck-saxl}, we deduce that $r=3$ and $S = \PSL_2(3^a)$ for odd $a > 1$. However, $\PSL_2(3) \cong A_4$ is a subgroup of $S = \PSL_2(3^a)$ that has order divisible by $2$ and $3$. This establishes that $S$ is not a minimal $\{\Z_2,\Z_r\}$-cover.

%In each case in the table we can see easily what $r$ would need to be, bearing
%in mind that $M$ must contain $N_G(R)$: if $G=A_p$ then $p=r\equiv 3$ (mod $4$);
%if $G=\PSL_2(p^a)$ then $a=1$ and $p=r\equiv 3$ (mod $4$); if $G=\PSL_d(p^a)$
%or $\PSU_d(p^a)$ then $r$ is a primitive prime divisor of $(p^a)^d-1$; $r$ is
%one of $23,21,47$ for the sporadic groups.
%
%In all cases we see that $r$ cannot be a Fermat prime, which proves that if $r$ is a Fermat prime then there are no simple minimal $\{\Z_2,\Z_r\}$-covers. 

Now assume that $r$ is not a Fermat prime, and let $d$ be an odd prime divisor of $r-1$. Let $p$ be a prime at least $5$ that has order $d$ modulo $r$, of which there are infinitely many options by Dirichlet's theorem on primes in arithmetic progressions, and let $G=\PSL_d(p)$. Then $r$ divides $(p^d-1)/(p-1)$ and so divides $|G|$, and if $R$ is a Sylow $r$-subgroup of $G$ then $R$ is cyclic and $N_G(R)$ has odd order. We claim that $N_G(R)$ is the only maximal subgroup of $G$ containing $x$, a non-trivial element of $R$.

To see this we use the work of Guralnick--Penttila--Praeger--Saxl on ppd-elements \cite{gpps}. We have set things up so that $x$ is a ppd element, and in the notation of \cite{gpps}, $x$ is a $\mathrm{ppd}(d,p;d)$-element. The possible subgroups of $\GL_d(p)$ containing $x$ are enumerated in \cite[Examples 2.1--2.9]{gpps}, and almost all of them can immediately be ignored since $d$ is a prime and $d$ appears twice in the phrase `$\mathrm{ppd}(d,p;d)$'. (The second $d$ in this is $e$ in \cite{gpps}, so in that paper we have that $d=e$ is prime.) We check each set of examples from \cite{gpps} in turn.
\begin{itemize}
\item Example 2.1 are classical groups, and since $d=e$ is an odd prime, and we are over a prime field, none of these applies.
\item Examples 2.2 and 2.3 do not apply since they require $e<d$.
\item Examples 2.4 and 2.5 do not apply as $d$ is an odd prime.
\item Example 2.6(a) does not apply since it requires $r-1=d$.
\item For Examples 2.6(b) and 2.6(c), there are three tables of examples; we need that $d=e$ is a prime, and the only option is $3\cdot A_7\leq \SL_3(25)$ with $r=7$. This is an example, but we required that $G$ be over a prime field, so this may be excluded.
\item For Example 2.7 there is a table of examples, and we find the options $G=M_{11}$, $(d,p,r)=(5,3,11)$ and $G=M_{23}$, $M_{24}$, $(d,p,r)=(11,2,23)$. Since we have (not coincidentally) chosen $p\geq 5$, these are not examples.
\item In Example 2.8, $e$ is always even, but in our case $e=d$ is odd.
\item In Example 2.9, there are no examples with $d$ a prime and $d=e$.
\end{itemize}

We thus find that $G=\PSL_d(p)$ is a minimal $\{\Z_2,\Z_r\}$-cover. As there are infinitely many options for $p$, we find infinitely many such groups.
\end{proof}

\begin{remark} \label{r:2p}
Using \cite{gpps} we could actually determine \emph{exactly} which simple groups are minimal covers of $\{\Z_2,\Z_r\}$. 
%It is already remarked in \cite{ls1991} that $A_r$, $r=7,11,23$ are not maximal odd-order subgroups, because of $\PSL_2(7)$, $M_{11}$ and $M_{23}$ respectively, so these do not appear in the list.
\end{remark}

We therefore have the following corollary.

\begin{corollary} \label{c:fermat}
If $r$ is a Fermat prime then the set $\{\Z_2,\Z_r\}$ has just three minimal covers, namely $\Z_{2r}$, $D_{2r}$ and $(\Z_2)^{2a}:\Z_r$ where $r=2^a+1$. 
%In particular, any finite group with order divisible by $2r$ contains one of the three subgroups above.
\end{corollary}

\begin{proof}
Combining Corollary~\ref{c:qr2} with Proposition~\ref{p:2p}, we see that there are no insoluble covers. Proposition~\ref{p:qr2}(a) shows that the three groups in the statement are the only possible soluble covers, and Cauchy's theorem establishes that they actually are covers.
\end{proof}

Propositions~\ref{p:qr}, \ref{p:35}, \ref{p:2p} and Corollary~\ref{c:fermat} together give Theorem~\ref{t:twocyclic}.

\begin{remark} \label{r:six-cfsg}
Our proofs use the Classification of Finite Simple Groups; however, the fact that there are only three $\{\Z_2,\Z_3\}$-covers (namely $\Z_6$, $S_3$ and $A_4$) can be proved without the Classification. Three papers~\cite{fss,gordon,pod} in 1977 independently determined the finite simple groups with no elements of order~$6$; and it is straightforward to show that, apart from the Suzuki groups (whose orders are not divisible by~$6$), they all contain subgroups isomorphic to
$S_3$ or $A_4$.
\end{remark}

\section{Cauchy numbers} \label{s:cauchy}

Building on the previous section, we now determine the Cauchy numbers, thus proving Theorem~\ref{t:cauchy}. Recall that $n$ is a \emph{Cauchy number} if there is a finite set $\W$ of groups of order divisible by $n$ such that if $n$ divides the order of $G$, then $G$ has a subgroup isomorphic to a group in $\W$. The main result of the previous section, Theorem~\ref{t:twocyclic}, asserts that if $n$ is the product of two distinct primes, then $n$ is a Cauchy number if and only if $n$ is twice a Fermat prime. 

We will use the
term ``$n$-group'' (for positive integer $n$) for a group whose order is
divisible by $n$, and ``$n$-witness'' for an $n$-group which has no
proper subgroups which are $n$-groups. So $n$ is a Cauchy number if and only
if the set of $n$-witnesses (up to isomorphism) is finite.

\begin{remark}
It was pointed out to us by Michael Kinyon that there is already a sequence of ``Cauchy numbers'', arising in the Laplace summation formula, see~\cite{msv}. ``Cauchy numbers'' also occur in the theory of compressible flow in continuum mechanics. However, we think that these topics are sufficiently different that no confusion will ensue.
\end{remark}

%We begin with a simple observation.
\begin{prop}
A divisor of a Cauchy number is a Cauchy number.
\label{p:divisor}
\end{prop}

\begin{proof}
It suffices to show that if $m$ is a positive integer and $p$ a prime such
that $pm$ is a Cauchy number, then $m$ is a Cauchy number. Let $G$ be an
$m$-witness. We claim that either $G$ or $G\times\Z_p$ is a
$pm$-witness. Suppose that $G$ is not a $pm$-group. Then $G\times\Z_p$ is a
$pm$-group. Let $H$ be a proper subgroup of $G\times\Z_p$ whose order is
divisible by $pm$, and consider the projection $\pi$ to the first factor of the
direct product. If $\pi$ restricted to $H$ is not a monomorphism, then
$H=K\times\Z_p$ for some $K<G$; then $K$ is not an $m$-group, so $H$ is not a
$pm$-group. Otherwise, $H\le G$, contrary to assumption. On the other hand,
if $G$ is a $pm$-group, it is clearly minimal.

Now let $\mathcal{W}_m$ be the set of $m$-witnesses, and define a map $F$
from $\mathcal{W}_m$ to $\mathcal{W}_{mp}$ given by
\[F(G)=\begin{cases}
G & \text{if $G$ is a $pm$-group,}\\
G\times\Z_p & \text{otherwise.}
\end{cases}\]
We claim that $F$ is one-to-one. By the Krull--Schmidt theorem, we need to show
that we cannot have $F(G_1)=G_1\times\Z_p=G_2=F(G_2)$. But if this holds, then
$G_1$ and $G_2$ are both $m$-witnesses, contradicting $G_1<G_2$.

Therefore, $|\mathcal{W}_m|\le|\mathcal{W}_{pm}|$, whence $|\mathcal{W}_m|$ is finite
and $m$ is a Cauchy number. 
\end{proof}

Our aim is to determine the Cauchy numbers. First we show that every positive integer satisfies the analogous condition when we restrict to soluble groups.

\begin{lemma}
If $G$ is an $n$-witness with normal subgroup $N$, then $G/N$ is an
$n/\gcd(n,|N|)$-witness.
\end{lemma}

\begin{proof}
$|G/N|$ is divisible by $n/\gcd(n,|N|)$. If it is not minimal with this
property, let $H$ be a proper subgroup of $G$ containing $N$ with $H/N$
divisible by $n/\gcd(n,|N|)$; then $n$ divides $|H|$, contrary to minimalty
of $G$.
\end{proof}

\begin{prop}
For any natural number $n$, there are only finitely many finite soluble groups
which are $n$-witnesses.
\label{p:soluble}
\end{prop}

\begin{proof}
The proof is by induction on $n$; the induction begins by noting that the
theorem is true when $n$ is a prime power. So assume that $n$ is not a 
prime power.

Let $N$ be a minimal normal subgroup of $G$. Then $N$ is an elementary abelian
$p$-group, for some $p$ dividing $n$. So $1<|N|<|G|$. By the inductive
hypothesis, $G/N$ is a $n/\gcd(n,|N|)$-witness, so its order is bounded. Also,
as $N$ is a minimal normal subgroup, it is generated by any $G/N$-orbit on
non-identity elements, so $|N|\le p^{|G/N|}$ is also bounded, and we are done.
\end{proof}

\begin{proof}[Proof of Theorem~\ref{t:cauchy}]
Combining Proposition~\ref{p:divisor} with Theorem~\ref{t:twocyclic}, we deduce that if $n$ is a Cauchy number which is not a prime power, then 
$n=2^bp^a$, where $p$ is a Fermat prime. The next step is to show that either
$n=6$, or $n=2p^a$ where $a\ge1$ and $p$ is a Fermat prime greater than $3$.
For this we have to exclude $n=18$ and $n=4p$ where $p$ is a Fermat prime (the
cases $p=3$ and $p=5$ require separate arguments).
\begin{enumerate}
\item For $n=12$, let $f$ be an odd prime, let
$q = 2^f$ and let $G = \PSL_2(q)$. Then $G$ is a $12$-witness. The maximal
subgroups of $G$ are, up to conjugacy, $D_{2(q+1)}$, $D_{2(q-1)}$, $2^f:(q-1)$.
Since $f$ is odd, we know that $q\equiv2$~(mod~$3$), so $3$ does not divide
$2(q-1)$ or $2^f:(q-1)$; and $4$ does not divide $2(q+1)$.
\item For $n=18$, again let $f$ be an odd prime, and take $q=3^f$ and
$G=\PSL_2(q)$. Then $G$ is an $18$-witness. For its only subgroups of order
divisible by $3$ are $\PSL_2(3)\cong D_6$ and $q:(q-1)/2$; and $(q-1)/2$ is odd.
\item For $n=20$, note that $A_5$ is a $20$-witness, and so 
Proposition~\ref{piFcovers} shows that $20$ is not a Cauchy number. 
\item Let $n=4p$, with $p$ a prime greater than $5$.
By the Chinese Remainder Theorem, there is an arithmetic progression of numbers
$r$ such that $r \equiv 1 \pmod{4}$ and $r \equiv -1 \pmod{p}$. Therefore, by
Dirichlet's Theorem, there are infinitely many such prime numbers $r$. Fix
such a prime $r$. Let $G = \PSL_2(r)$, which has order $\frac{1}{2}(r-1)r(r+1)$.
We claim that $G$ is an $n$-witness. First note that $n$ divides $|G|$, since
$4$ divides $r-1$ and $p$ divides $\frac{1}{2}(r+1)$. We now claim that $n$
divides the order of no proper subgroup of $G$. For a contradiction, suppose
that $H$ is a maximal subgroup of order divisible by $n$. Since $p > 5$, we
know that $H \neq A_5$. Since $p > 2$ and $p$ divides $r+1$, we know that $p$
does not divide $r-1$. Therefore, consulting the list of maximal subgroups of
$\PSL_2(r)$, we see that $H \cong D_{r+1}$, but $4$ does not divide $r+1$ since $r \equiv 1 \pmod{4}$. This contradicts the fact that $n$ divides $|H|$.
\end{enumerate}

Now we have proved that a number not of the form described in the theorem is
not a Cauchy number. So from now on we assume that $n=2p^a$, where $p$ is a
Fermat prime greater than $3$ and $a$ a positive integer, and have to show
that $n$ is a Cauchy number. By Proposition~\ref{p:soluble} it suffices to show that all $n$-witnesses are soluble. 

Assume that there exists an insoluble $n$-witness, and let $G$ be one with $a$
chosen minimal, and subject to this $|G|$ minimal. Let $N$ be a proper nontrivial normal subgroup of $G$. Since $G/N$ is a $n/\gcd(n,|N|)$-witness, the
minimality of $n$ and $G$ shows that $G/N$ is soluble.

We claim that $G/N$ is a $p$-group. To see this, first note that $N$ is
insoluble since $G/N$ is soluble, so, in particular, $|N|$ is even. Fix
$N \leq P \leq G$ such that $P/N$ is a Sylow $p$-subgroup of $G/N$. Then
$n = 2p^a$ divides $|P|$. Since $G$ is a $n$-witness, we deduce that $G = P$.
Therefore, $G/N = P/N$ is a $p$-group.

Since every proper quotient of $G$ is a $p$-group, $O^p(G)$ is the unique
minimal normal subgroup of $G$. Write $O^p(G) = S^k$ where $S$ is a nonabelian
simple group and $k \geq 1$.

Let $P$ be a Sylow $p$-subgroup of $S$ and let $H$ be a maximal subgroup of $G$
containing $N_G(P^k)$. Let $\widehat{P}$ be a Sylow $p$-subgroup of $G$
containing $P^k$. Then $\widehat{P} \leq N_G(P^k) \leq H$ since
$P^k = \widehat{P} \cap S^k \leqn \widehat{P}$. Therefore, $p^a$ divides $|H|$.
In particular, $|H|$ is odd since $G$ is a $2p^a$-witness, and $S^k \not\leq H$
since $H$ is a proper subgroup of $G$. In particular, the action of $G$ on
$G/H$ is faithful and primitive and the point stabiliser $H$ has odd order.
Therefore, the main corollary in \cite{ls1991} restricts the possibilities for
$G$ and $H$. Namely, if we identify $S$ with the first factor of $S^k$, then
$H \cap S = M \cap S$ where $M$ is an odd-order maximal subgroup of an almost
simple group with socle $S$. However, $N_S(P) \leq N_G(P^k) \leq H$, so
$N_S(P) \leq M \cap S$, which means that $|H|$ is even, by
Lemma~\ref{l:liebeck-saxl}. This contradiction completes the proof. 
\end{proof}

We conclude this section by commenting on the set of $n$-witnesses when $n$ is a Cauchy number. If $n = p^a$, then the $n$-witnesses are the groups of order $n$, and if $n = 2p$, for a Fermat prime $p$, then Corollary~\ref{c:fermat} implies that the $n$-witnesses are $\Z_{2p}$, $D_{2p}$, $\Z_2^{2a} : \Z_p$ where $p=2^a+1$. The following result handles the remaining case.

\begin{prop} \label{p:witness}
Let $n = 2p^a$ for a Fermat prime $p > 3$ and $a > 1$. 
\begin{enumerate}
\item Let $P$ be a group of order $p^a$ and let $V$ is an irreducible $\mathbb{F}_2P$-module. Then $V{:}P$ is a $n$-witness.
\item Any $n$-witness has order dividing $2^dp^a$ where $d$ is the maximum degree of an representation of group of order $p^a$ over $\mathbb{F}_2$.
\end{enumerate}
\end{prop}

\begin{proof}
First consider part~(a). Let $G = V{:}P$. Clearly $n = 2p^a$ divides $|G|$. The maximal subgroups of $G$ are $P$, of order $p^a$, and $V{:}H$ where $H$ is a maximal subgroup of $P$, of order $2p^{a-1}$. This proves that $G$ is an $n$-witness.

Now consider part~(b). Corollary~\ref{c:fermat} gives the result when $a=1$. Now assume that $a > 1$ and proceed by induction on $a$. Let $G$ be an $n$-witness. We know that $G$ is soluble and hence a $\{2,p\}$-group, by Hall's theorem. Let $M$ be a minimal normal subgroup of $G$. If $M$ is a $2$-group, then by minimality, $G/M$ must be a group of order $p^a$, and since $M$ is a minimal normal subgroup, $G/M$ must act irreducibly on $M$. Now assume that $M$ is a $p$-group. Suppose that $|M| \geq p^a$. Then by minimality, $|G/M| = 2$. However, $M$ is a minimal normal subgroup, so $G/M$ must act irreducibly on $M$, which forces $|M| = p$, so $a = 1$, contrary to our assumption. Therefore, $|M| < p^a$, so, by minimality, $G/M$ is a $2p^a/|M|$-witness, and the claim holds by induction. 
\end{proof}

\section{Simple groups} \label{s:simple}

Let us now turn to covers of finite simple groups, beginning with the following general result.

\begin{theorem}\label{hrd82}\label{t:composition}
Let $\F$ be a finite set of finite simple groups of size $n$ and $G$ be  a minimum $\mathcal{F}$-cover. Let  $N_0<N_1<\cdots<N_k=G$ be a composition series of $G$ and $\F_i=\{H\in\F:H$ is isomorphic to a subgroup of $N_i/N_{i-1}\}$. Then
\begin{enumerate}
\item $k\leq n$.
\item   $\bigcup_{i=1}^k \F_i=\F$.
\item $ N_i/N_{i-1}$ is a minimum $\F_i$-cover.
\item   $\Pi_{i=1}^k N_i/N_{i-1}$ is a minimum $\mathcal{F}$-cover.
\end{enumerate}
\end{theorem}

\begin{proof}
The proof is by induction on $k$. For $k=1$ all the statements are trivial. Now assume $k\geq 2$. Let $\F'=\{H\in\F:H$ is isomorphic to a subgroup of $N_{k-1}\}$. Then every simple subgroup of $G$ is isomorphic to a subgroup of $N_{k-1}$ or a subgroup of $G/N_{k-1}$. So $\F=\F'\cup \F_k$.
Since $G$ is a minimum $\mathcal{F}$-cover, $\F'$ and $\F_k$ are nonempty. Also $N_{k-1}$ is an $\mathcal{F}'$-cover and $G/N_{k-1}$ is an $\mathcal{F}_k$-cover. If $N_{k-1}$ is not a minimum $\mathcal{F}'$-cover and $M$ is a minimum $\mathcal{F}'$-cover then $M\times G/N_{k-1}$ is an $\mathcal{F}$-cover whose order is less than $|G|$ which is a contradiction. Similarly $G/N_{k-1}$  is a minimum $\mathcal{F}_k$-cover.   So   by induction $k-1\leq |\F'|\leq n-1$ which implies $k\leq n$. Also by induction  $\bigcup_{i=1}^{k-1} \F_i=\F'$  and $ N_i/N_{i-1}$ is a minimum $\mathcal{F}_i$-cover  for $1\leq i\leq k-1$ and  $\Pi_{i=1}^{k-1} N_i/N_{i-1}$ is a minimum $\mathcal{F}'$-cover. So  $\Pi_{i=1}^k N_i/N_{i-1}$ is a minimum $\mathcal{F}$-cover.
\end{proof}

We now consider Theorem~\ref{t:simple}, which we repeat below.

\begin{thm-simple}
Let $M$ and $N$ be nonabelian finite simple groups. If $G$ is a minimum cover of $\{ M, N \}$ then either $G = M \times N$ or $G$ is simple and $|G| \leq |M| \cdot |N|$.
\end{thm-simple}

In particular, if there is an $\{ M, N \}$-cover of order less than $|M|\cdot|N|$, then any minimum $\{ M, N \}$-cover is simple.

Before proving Theorem~\ref{t:simple}, let us note that both possibilities in the theorem can occur as the following two examples demonstrate. These examples can be verified using the \emph{$\mathbb{ATLAS}$ of Finite Groups} \cite{ATLAS}.

\begin{example} \label{ex:simple_1}
Let $M=A_5$ and $N=\PSL_2(8)$. The orders of these groups are $60$ and $504$.
Their least common multiple is $2520$ and their product is $30240$. The only
simple groups with order divisible by $2520$ and not greater than $30240$ are
$A_7$, $A_8$ and $\PSL_3(4)$; none of these embed $\PSL_2(8)$. By the theorem,
the unique minimum $\{M,N\}$-cover is $M\times N$.
\end{example}

\begin{example} \label{ex:simple_2}
Let $M=A_6$ and $N=\PSL_2(7)$. Their orders are $360$ and $168$, with least
common multiple $2520$. There is a unique simple group of order $2520$,
namely $A_7$, which embeds both $M$ and $N$. Therefore, $A_7$ is the unique minimum
$\{M,N\}$-cover.
\end{example}

\begin{proof}[Proof of Theorem~\ref{t:simple}]
By Theorem~\ref{t:composition}, either $G$ is simple, or it has a composition
series of length~$2$ with composition factors $M$ and $N$.

Suppose, without loss of generality, that $G$ has a normal subgroup
isomorphic to $M$ with quotient isomorphic to $N$. Hence $C_G(M)\lhd G$ and $M\cap C_G(M)=Z(M)=\{1\}$. Each element of $G$ acts
on $M$ by conjugation. A consequence of the Classification of Finite Simple
Groups is that the outer automorphism group of $M$ is soluble. Since $G$ has
no non-trivial soluble quotient, we see that each element of $G$ induces an
inner automorphism of $M$.  Let $g\in G$. Then  there exists $m\in M$ such that for all $x\in M$, we have $gxg^{-1}=mxm^{-1}$. So $m^{-1}g\in C_G(M)$ which implies $g\in MC_G(M)$. Hence   $G=MC_G(M)$.
 Thus $G$
has normal subgroups $M$ and $C_G(M)$ intersecting trivially (and commuting),
so is their direct product.
\end{proof}

\begin{corollary} \label{c:large_and_small}
There is a function $f$ such that, if $\F=\{M,N\}$ where $M$ and $N$ are
nonabelian finite simple groups with $|N|>f(|M|)$, then $M\times N$ is the
unique minimum $\F$-cover.
\end{corollary}

\begin{proof}
Suppose not, and let $G$ be a minimum $\F$-cover. Then $G$ is simple, and has
a subgroup $N$ with index at most $|M|$. Now $G$ acts faithfully on the
cosets of $N$, and so it is embeddable in the symmetric group of degree
$|M|$, with $N$ as the point stabiliser. So $N$ is embeddable in the symmetric
group of degree $|M|-1$, and $|N|\le(|M|-1)!$.
\end{proof}

We conclude this section by asking: Is it possible for a set of two nonabelian finite simple groups to have two nonisomorphic minimum covers?

In light of Theorem~\ref{t:simple}, if $M$ and $N$ are nonabelian finite simple groups and $\F=\{M,N\}$ has two nonisomorphic minimum covers, then one of the following must hold:
\begin{enumerate}
\setlength\itemsep{0pt}
\item there are two simple groups of the same order (smaller than $|M|\cdot|N|$) which are minimum $\F$-covers; or
\item there is a simple group of order $|M|\cdot|N|$ which is a minimum $\F$-cover.
\end{enumerate}

For (a), with the Classification of Finite Simple Groups, the only pairs of nonisomorphic finite simple groups of the same order are $\{\PSL_3(4),A_8\}$ and $\{\mathrm{PSp}_{2m}(q),\mathrm{P}\Omega_{2m+1}(q)\}$ for $m\ge3$ and $q$ odd. Using the $\mathbb{ATLAS}$~\cite{ATLAS}, we can show that the first pair are not both minimal covers of any pair of simple groups. We suspect that there are no examples for the second pair either.

For (b), the following question arises which is of independent interest.

\begin{question} \label{q:simple_product} 
Classify the triples $(M,N,G)$ of nonabelian finite simple groups such that $|M|\cdot|N|=|G|$.
\end{question}

If $G$ is a finite simple group that has a sharply $t$-transitive action of degree $n$, then $|A_{n-t}| \cdot |G| = |A_n|$. The following result addresses this special case.

\begin{prop} \label{p:simple_product}
Let $G$ be a finite simple group with a sharply $t$-transitive action of degree $n$. Then $A_n$ is a minimum cover of $\{A_{n-t},G\}$ if and only if one of the following holds
\begin{enumerate}
\item $G = \PSL_2(2^f)$ and $(n,t) = (2^f+1,3)$ where $f \geq 3$
\item $G = {\rm M}_{12}$ and $(n,t) = (12,5)$
\end{enumerate}
\end{prop}

\begin{proof}
First assume that $t = 1$, so $G$ acts regularly on $n$ points. Since $G$ has a faithful action on strictly fewer points, $G$ embeds in $A_{n-1}$. In particular, $A_{n-1}$ is a minimum cover of $\{A_{n-1},G\}$.

Now assume that $t > 1$. Consulting \cite[Theorem~4.11]{Cameron99}, for example, the only sharply $t$-transitive actions of a finite simple group $G$ are:
\begin{itemize}
\setlength\itemsep{0pt}
\item[1.] $G = \PSL_2(2^f)$ and $(n,t) = (2^f+1,3)$ where $f \geq 3$
\item[2.] $G = {\rm M}_{11}$ and $(n,t) = (11,4)$
\item[3.] $G = {\rm M}_{12}$ and $(n,t) = (12,5)$
\end{itemize}

It is straightforward to rule out Case~2: ${\rm M}_{23}$ is a cover of $\{ A_7, {\rm M}_{11} \}$ and $|{\rm M}_{23}| = 10200960 < 19958400 = |A_{11}|$. For Case~3, it is also easy to verify (in \textsc{Magma} \cite{Magma}, say) that $A_{12}$ is a minimum cover of $\{ A_7, {\rm M}_{12} \}$. (For comparison with Case~2, while ${\rm M}_{24}$ is a cover of $\{ A_7, {\rm M}_{12}\}$ we have  $|{\rm M}_{24}| = 244823040 > 239500800 = |A_{12}|$.)

Case~1 remains. Fix $f \geq 3$, write $q=2^f$ and let $G = \PSL_2(q)$. We claim that $A_{q+1}$ is the smallest simple group to embed $A_{q-2}$ and $G$. If $f \in \{ 3, 4 \}$, then this is easily verified in \textsc{Magma} \cite{Magma}, so let us assume that $f \geq 5$. Let $H$ be a simple group embedding both $A_{q-1}$ and $G$. Since $H$ embeds $\PSL_2(q)$, if $H = A_d$, then $d \geq q+1$ (see \cite[Theorem~5.2.2]{KleidmanLiebeck90}). Since $H$ embeds $A_{q-2}$ with $q-2 \geq 30$, we deduce that $H$ is not a sporadic group (see \cite[Theorem~5.2.9]{KleidmanLiebeck90}).  Now assume that $H$ is a classical group. Since $H$ embeds $A_{q-2}$, the dimension of the natural module for $H$ must be at least $q-4$ (see \cite[Theorem~5.3.7]{KleidmanLiebeck90}). Consulting the order formulae for these groups, it is easy to see that this implies that $|H| \geq 2^{(q-4)(q-6)/2}$. This means
\[
\log_2|H| \geq \tfrac{1}{2} (q-4)(q-6) \geq q \log_2(q+1) \geq \log_2 |A_{q+1}|.
\] 
Finally assume that $H$ is an exceptional group of Lie type. In light of the previous cases, consulting the possible maximal subgroups of $H$ \cite[Theorem~8]{LiebeckSeitz03}, we see that $H$ does not embed $A_{q-2}$ with $q-2 \geq 30$. Therefore, in all cases, $|H| \geq |A_{q+1}|$, so $A_{q+1}$ is a minimum cover of $\{A_{q-2},G\}$, as claimed.
\end{proof}

%Precover:=function(a,m,g,X)
%  A:=Group(a);
%  M:=Group(m);
%  l:=LCM(Order(A),Order(M)); 
%  k:=SimpleGroupNameToNumber(g);
%  I:=[];
%  for i in [1..k] do
%    if (Order(SimpleGroup(i)) mod l eq 0) and (not i in X) then
%      I:=Append(I,i);
%      SimpleGroupName(i);
%    end if;
%  end for;
%  return I;
%end function;
%Precover("A7","M11","A11",[]);
%Precover("A7","M12","A12",[]);
%X:=[1,3,8,20,38,64,114,201,364,711,1468,3188,7244,17099,41577,103724];
%Precover("A6","L2(8)","A9",X);
%Precover("A14","L2(16)","A17",X);

\begin{corollary} \label{c:simple_product}
A set of two nonabelian finite simple groups can have two nonisomorphic minimum covers. For example,
\begin{enumerate}
\item $A_{2^f+1}$ and $A_{2^f-2} \times \PSL_2(2^f)$ are minimum covers of $\{ A_{2^f-2}, \PSL_2(2^f) \}$ whenever $f \geq 3$
\item $A_{12}$ and $A_7 \times {\rm M}_{12}$ are minimum covers of $\{ A_7, {\rm M}_{12} \}$.
\end{enumerate}
\end{corollary}

\begin{remark} \label{r:prime_divisors}
In an earlier version of the paper, we asked whether any prime divisor of the order of a minimum cover of a set of finite groups must divide the order of one of the groups in the set. However, this is false, since $M_{23}$ is the unique minimum cover of $\{\PSL_3(4),A_8\}$, two simple groups of the same order $20160$ (this can be seen from the $\mathbb{ATLAS}$~\cite{ATLAS}), but $|M_{23}|$ is divisible by $11$ and $23$, neither of which divide $20160$.
\end{remark}

\section{Groups of prime-power order} \label{s:fp}

In this section we examine sets of $p$-groups, for $p$ prime.

\begin{remark}\label{hrd7}
It is well known that the only groups of order $8$ are $\Z_8, \Z_{4}\times\Z_{2},\Z_{2}\times\Z_{2}\times\Z_{2},D_8,Q_8$. So the only $4$-cover groups of order $8$ are $\Z_{4}\times\Z_{2}$ and $D_8$.
Similarly there are five groups of order $p^3$ for an odd prime $p$ which are $\Z_{p^3}, \Z_{p^2}\times\Z_{p},\Z_{p}\times\Z_{p}\times\Z_{p}$, and the two
non-abelian groups
\begin{eqnarray*}
G_p
&=&\left\{\left[
 \begin{array}{ll}
 a & b\\
 0 & 1
 \end{array}\right]:a\in  1+p\Z_{p^2},b\in\Z_{p^2}\right\},\\
\hbox{and }\Heis(p)
&=&\left\{\left[
 \begin{array}{lll}
 1 & a& b\\
 0 & 1&c\\
 0 & 0& 1
 \end{array}\right]:a,b,c\in \Z_{p}\right\}
\end{eqnarray*}
with exponent $p^2$ and $p$ respectively.
This implies the only $p^2$-cover groups of order $p^3$ are $G_p$ and $\Z_p\times\Z_{p^2}$
\end{remark}

\begin{prop}\label{hrd8}
Let $p$ be a prime number. Then
\begin{enumerate}
\item for $p=2$ the only minimal $4$-cover groups are $\Z_{4}\times\Z_{2}$ and $D_8$;
\item for $p>2$ then  the only minimal $p^2$-cover groups are $\Z_{p^2}\times\Z_{p}$ and $G_p$.
\end{enumerate}
\label{p:p2}
\end{prop}

\begin{proof}
Let $G$ be a minimal $p^2$-cover group. Then $G$ is a $p$-group by Proposition~\ref{p:pgroup}. It is well known that $\Z_{p^2}$ and  $\Z_{p}\times\Z_{p}$ are the only two group of order $p^2$ up to isomorphism. Let $H$ and $K=\langle a\rangle$ be subgroups of $G$ which are isomorphic to  $\Z_p\times\Z_{p}$ and
$\Z_{p^2}$. Also $G$ contains a normal subgroup $N$ of order $p^2$. If $N$ is cyclic then there is an element $g\in H\backslash N$.
 Then $\langle g\rangle N$ is a  $p^2$-cover of order $p^3$ and the proof is complete by Remark~\ref{hrd7}. So
assume $N\cong \Z_p\times\Z_p$.  First assume $|N\cap K|=1$. Then $NK$ is a $p^2$-cover of order $p^4$. Hence $G=NK$. Since $N_G(N)/C_G(N)$ is isomorphic to a subgroup of $\Aut(N)\cong\GL(2,p)$, so $|N_G(N)/C_G(N)|\mid p$. Thus $a^p\in C_G(N)$. So $a^p\in Z(G)$. Also $N\cap Z(G)$ is non-trivial so there is an element $g\in Z(G)\cap N$ of order $p$. Hence $\langle g\rangle K$   is a  $p^2$-cover of order $p^3$ which is a contradiction by minimality of $G$. Hence $|N\cap K|=p$. So $|NK|=p^3$ and $NK$ is a $p^2$-cover. So the proof is complete by Remark~\ref{hrd7}.
\end{proof}

In the next case, we have the following.

\begin{theorem}\label{thm:2covers} \quad
\begin{enumerate}
\item There are two minimum $2^3$-cover groups, both of order $2^5$.
\item For prime $p>2$, there is no $p^3$-cover of order $p^5$, but there is
one of order $p^6$; so a minimum $p^3$-cover has order $p^6$.
\end{enumerate}
\end{theorem}

\begin{proof}
Part (a) can be proved by computer: we used the computer algebra system
\textsf{GAP}~\cite{gap}. The two $8$-covers are the groups 
\texttt{SmallGroup(32,40)} and
\texttt{SmallGroup(32,43)} in the \textsf{GAP} library. (There is no $8$-cover
of order~$16$. For suppose that $G$ was an $8$-cover of order~$16$. Then $G$
contains subgroups $A\cong\Z_8$ and $B\cong(\Z_2)^3$; clearly $|A\cap B|\le 2$,
so $|AB|=|A|\cdot|B|/|A\cap B|\ge32$.)

For (b), the case $p=3$ can also be shown using \textsf{GAP}. For $p\ge5$, we
proceed as follows. Let $G$ be a $p^3$-cover of order $p^5$. Then $G$ has
nilpotency class at most~$4$, so smaller than $p$; hence $G$ is a regular
$p$-group \cite[p.~183]{hall}. Now \cite[Theorem 12.4.5]{hall} shows that
the elements of order $p$ in $G$, together with the identity, form a subgroup
$H$ of $G$. Now since $G$ contains both the elementary abelian group $(\Z_p)^3$
of order $p^3$ and the non-abelian group $\Heis(p)$ of order $p^3$ and
exponent~$p$; so the subgroup $H$ must satisfy $|H|>p^3$, so $|H|\ge p^4$.
Now $G$ must also contain the cyclic group
$K=\Z_{p^3}$, and $|H\cap K|\le p$, so $|HK|\ge p^4\cdot p^3/p=p^6$. So no
$p^3$-cover of order $p^5$ can exist.

For the example, we start with the group $E=\Heis(p)$ of order $p^3$ and
exponent $p$:
\[E=\langle x,y,z\mid x^p=y^p=z^p=[x,z]=[y,z]=1,[x,y]=z\rangle.\]
It is easy to see that $E$ has an automorphism $\alpha$ such that
$\alpha(x)=x$ and $\alpha(y)=xy$.

Now we use the following result~\cite[Theorem 15.3.1]{hall}:

\begin{lemma}
Let $N$ be a finite group, $\alpha\in\Aut(N)$, and $m\in\mathbb{N}$. Then the following assertions are equivalent:
\begin{enumerate}
\item There exists a finite group $H$ such that $N\unlhd H$,
$H/N = \langle hN\rangle =\Z_m$ and $\alpha(x) = hxh^{-1}$ for all $x\in N$;
\item There exists $n\in N$ such that $\alpha(n) = n$ and
$\alpha^m(a) = nan^{-1}$ for all $a\in N$.
\end{enumerate}
\end{lemma}

We apply the lemma with $m=p^2$ and $n=z$, giving a non-split extension $H$ of
$E$ such that $H/E = \langle aE\rangle\cong\Z_{p^2}$ with $a^{p^2} = z$ and
$e^a = \alpha(e)$ for all $e\in E$. For $b:= a^p y^{-1}\in H$ we compute
$b^x = a^p[x,y]y^{-1} = bz = ba^{p^2} = b^{1+p}$. Hence,
$H$ has subgroups $E\cong\Heis(p)$, $\langle a\rangle\cong\Z_{p^3}$,
$\langle a^p,x\rangle\cong\Z_{p^2}\times\Z_p$ and 
$\langle b,x\rangle\cong G_p$.

Finally, set $G=H\times\Z_{p}$; it is clear that $G$ also contains $(\Z_p)^3$.
\end{proof}

In addition, computation with \textsf{GAP} shows that, for $p=3$, there are
many examples of $p^3$-covers of order $p^6$.

\medskip

Next we show that there are infinitely many minimal $2^3$-covers, and that
these may be taken to be strongly minimal; indeed, having no proper
subquotient which is a $2^3$-cover.

For the proof we use the \emph{semidihedral group}
\[SD_{2^n}=\langle a,b\mid a^{2^{n-1}}=b^2=1,b^{-1}ab=a^{2^{n-2}-1}\rangle\]
of order $2^n$, for $n\ge4$. We make a couple of observations about this
group.
\begin{itemize}
\item Its centre is cyclic of order $2$, generated by $a^{2^{n-2}}$.
\item $(ab)^2=a\cdot a^{2^{n-2}-1}=z$, so $ab$ has order $4$.
\item It is a $2$-generated $2$-group, and so has three maximal
subgroups of index $2$. These are $\langle a\rangle$ (cyclic), 
$\langle a^2,b\rangle$ (dihedral) and $\langle a^2,ab\rangle$ (generalized
quaternion).
\end{itemize}

\begin{theorem}
Let $n\ge4$, $A=SD_{2^n}$, and $C=\langle c\rangle=\Z_2$. Then the group
$A\times C$ of order $2^{n+1}$ is a minimal and co-minimal $8$-cover.
\label{t:inf8}
\end{theorem}

\begin{proof}
First we show that all groups of order $8$ are subgroups of $G$. We have
$\langle a^{2^{n-4}}\rangle\cong\Z_8$,
$\langle a^{2^{n-3}},c\rangle\cong\Z_4\times \Z_2$,
$\langle a^{2^{n-2}},b,c\rangle\cong(\Z_2)^3$,
$\langle a^{2^{n-3}},b\rangle\cong D_8$, and
$\langle a^{2^{n-3}},ab\rangle\cong Q_8$.

Now we show that no proper subgroup of $G$ is an $8$-cover. It suffices to
consider a maximal subgroup $H$. Let $\phi$ be the projection of $A\times C$
onto the first factor. If $c\notin H$, then the restriction of $\phi$ to $H$
is an isomorphism to $A$, so $H$ is semidihedral. If $c\in H$, then
$H=K\times C$ where $K$ is a maximal subgroup of $A$, and so is cyclic,
dihedral or generalized quaternion. None of these groups is an $8$-cover.

Finally we show that no proper quotient of $G$ is a $2$-cover. Again it
suffices to consider maximal quotients $G/N$, where $N$ is a normal subgroup
of $G$. Then $N\le Z(G)=\langle z,c\rangle$; so $N=\langle z\rangle$,
$\langle c\rangle$ or $\langle zc\rangle$.  If $z\notin N$, then the 
restriction of the projection $G\to G/N$ to $A$ is an isomorphism, so $G/N$ is
semidihedral. Otherwise $G/N=(A/\langle z\rangle)\times C
\cong D_{2^{n-1}}\times \Z_2$. Again none of these
groups, or any of their quotients, is an $8$-cover.
\end{proof}

\begin{remark} A slightly more elaborate argument shows that in fact no
quotient of a subgroup of $G$, apart from $G$ itself, is an $8$-cover.
\end{remark}

The numbers of $8$-covers of order $2^n$ are given below, together with the numbers of minimal and strongly minimal $8$-covers, for $5\leq n\leq 8$.

\begin{center}\begin{tabular}{ccccc}
\hline Order & $32$&$64$&$128$&$256$
\\\hline  Number of groups&$51$&$267$&$2328$&$56092$
\\ Number of $8$-covers&$2$&$45$&$745$&$14798$
\\ Number of minimal $8$-covers&$2$&$18$&$85$&$969$
\\ Number of strongly minimal $8$-covers&$2$&$14$&$3$&$7$
\\ \hline
\end{tabular}\end{center}

\begin{remark}
A minimum $16$-cover has order $2^8$, and \texttt{SmallGroup(256,384)} in the
\textsf{GAP} library is an example.
\end{remark}

\medskip

In contrast to the upper bound for the order of a minimum $p^m$-cover (the order
of the Sylow subgroup of $S_{p^m}$), we give a lower bound $p^{\Omega(m^2)}$,
which is probably rather weak.

We begin with a brief note. The fraction $|\GL(n,p)|/p^{n^2}$ is the 
probability that an $n\times n$ matrix over the field of order $p$ is
invertible. It can be written as
\[\prod_{i=1}^{n}(1-p^{-i}).\]
The theory of infinite products shows that, as $n\to\infty$ with $p$ fixed,
it decreases to a positive limit $\theta(p)$, which is an evaluation of a
Jacobi theta-function. It is easily seen that $\theta(p)$ is an increasing
function of $p$. The value of $\theta(2)$ is $0.2887\dots$. So the
probability that an $n\times n$ matrix over the $p$-element field is invertible is at
least $\theta(2)$ for any $n$ and $p$.

\begin{lemma}
Let $G$ be a group of order $p^n$. Then the number of $n$-tuples of elements of $G$, which generate $G$ is at least $cp^{n^2}$.
\label{l:ntuples}
\end{lemma}

\begin{proof}
Let $|G/\Phi(G)|=p^k$.
By Burnside's basis theorem, a $k$-tuple $g_1,\ldots,g_k$ generates $G$ if
and only if the images of $g_1,\dots,g^k$ in $G/\Phi(G)$ form a basis for
this quotient, which is isomorphic to a $k$-dimensional vector space over
the $p$-element field. The number of such bases is the order of $\GL(k,p)$,
which as noted is at least $cp^{k^2}$. For each basis element, there are
$p^{n-k}$ elements of the corresponding coset of $\Phi(G)$ in $G$. Also, we
can complete the $n$-tuple by choosing arbitrary elements of $G$, each in $p^n$
ways. So the number of $n$-tuples is
\[|\GL(k,p)|\cdot p^{k(n-k)}\cdot p^{(n-k)n} \ge cp^{n^2},\]
as required.
\end{proof}

\begin{theorem}
The order of a minimum $p^n$-cover is at least $p^{(2/27+o(1))n^2}$.
\label{t:pncover}
\label{t:227}
\end{theorem}

\begin{proof}
Suppose that $G$ is a minimum $p^n$-cover, of order $p^N$. There are $p^{Nn}$
$n$-tuples of elements of $G$; among them are generating tuples for all groups
of order $p^n$. By Lemma~\ref{l:ntuples}, each group of order $p^n$ has at least
$cp^{n^2}$ generating $n$-tuples. So
the number of groups of order $p^n$ is at most
\[p^{nN}/(cp^{n^2})=c^{-1}p^{n(N-n)}.\]

However, it was proved by Higman and Sims~\cite{higman,sims} (see also
\cite{bnv}) that the number of different
groups of order $p^n$ is $p^{(2/27+o(1))n^3}$. So
\[c^{-1}p^{n(N-n)}\ge p^{(2/27+o(1))n^3},\]
from which we find that $N\ge(2/27+o(1))n^2$.
\end{proof}

\begin{question}
Find better bounds for the order of a minimum $p^n$-cover. In particular, is
there an upper bound of the form $p^{F(n)}$, where $F$ is independent of $p$?
\end{question}

In the next section, we will find the smallest abelian group which contains
all abelian groups of order $p^n$; its order is roughly $p^{n\log n}$.

\section{Cyclic, abelian and nilpotent groups} \label{s:can}

The observation that $A_5$ is a minimum cover for
$\{\Z_3,(\Z_2)^2,\Z_5\}$ shows that it is not true that, if all groups
in $\F$ are abelian, nilpotent, or soluble, then every minimum
$\F$-cover has the same property. Moreover, $\{\Z_2,\Z_3\}$ has two minimum
covers, $\Z_6$ and $S_3$. So the best we can hope is that, if all groups in
$\F$ have a certain property, then at least one minimum $\F$-cover has this
property. This is the case for cyclic groups:

\begin{theorem}
Let $n_1,\ldots,n_k$ be positive integers and $N=\lcm(n_1,\ldots,n_k)$. Then
$\Z_N$ is a minimum cover for $\F=\{\Z_{n_1},\ldots,\Z_{n_k}\}$.
\end{theorem}

\begin{proof}
Clearly $\Z_N$ is an $\F$-cover, and by Proposition~\ref{p:order} it is
minimum.
\end{proof}

Perhaps the next simplest example of Question~\ref{q:property} is one which we have
not been able to settle:

\begin{question}
Let $\F$ be a set of abelian $p$-groups, for some prime $p$. Is there a
minimum $\F$-cover which is an abelian $p$-group?
\label{q:apgp}
\end{question}

\begin{theorem} \label{t:can}
Suppose that $\F$ is a finite set of finite nilpotent groups. Then there
is a minimum $\F$-cover which is nilpotent. If Question~\ref{q:apgp} has an
affirmative answer, then the same holds with ``abelian'' replacing
``nilpotent''.
\end{theorem}

\begin{proof}
Let $\F=\{F_1,\ldots,F_r\}$ be a finite set of finite nilpotent groups,
and let $G$ be a minimum $\F$-cover. For each prime $p$, let
$F_i(p)$ be the Sylow $p$-subgroup of $F_i$, and let
$\F(p)=\{F_1(p),\ldots,F_r(p)\}$. Let $H(p)$ be a minimum $\F(p)$-cover, and
note that $H(p)$ is a $p$-group, by Proposition~\ref{p:pgroup}. Let
$H$ be the direct product of the groups $H(p)$.
Now a Sylow $p$-subgroup $G(p)$ of $G$ is an $\F(p)$-cover, so
$|G(p)|\ge|H(p)|$, and thus $|G|\ge|H|$. But since each group in $\F$ is
the direct product of its Sylow subgroups, it is embeddable in $H$, and
thus by minimality $|G|=|H|$. So $H$ is a nilpotent cover of $\F$ with
smallest possible order.

Now suppose that all the groups in $\F$ are abelian, and that 
Question~\ref{q:apgp} has an affirmative answer. The same argument then
applies, using the fact
that a nilpotent group is abelian if and only if all its Sylow subgroups are.
\end{proof}

What is the order of the minimum cover? The theorem shows that it is enough to
find the order of a minimum cover of a set of $p$-groups. We can answer
this question in the case of abelian $p$-groups, again assuming that
Question~\ref{q:apgp} has an affirmative answer.

Suppose that $\F=\{F_1,\ldots,F_r\}$ is a set of abelian $p$-groups. We can
write each one in canonical form:
\[F_i=\Z_{p^{a(i,1)}}\times\cdots\times\Z_{p^{a(i,k)}},\]
where $a(i,1)\ge a(i,2)\ge\cdots\ge a(i,k)$; by adding extra zero terms if
necessary we can assume that the value of $k$ is the same for each group. Let
\[c(j)=\max\{a(1,j),a(2,j),\ldots,a(r,j)\}\]
for $j=1,\ldots,k$. We claim that
\[c(1)\ge c(2)\ge\cdots\ge c(k).\]
For suppose that $c(j+1)=a(i,j+1)$. Then $c(j)\ge a(i,j)\ge a(i,j+1)=c(j+1)$.

Let
\[P=\Z_{p^{c(1)}}\times\cdots\times\Z_{p^{c(k)}}.\]
The above claim shows that this is the canonical form for $P$.

\begin{prop}
With the above notation, $P$ is the smallest abelian $\F$-cover.
\label{p:minimum}
\end{prop}

The proof depends on the following lemma:

\begin{lemma}
Let 
\[A=\Z_{p^{a(1)}}\times\cdots\times\Z_{p^{a(k)}}\hbox{ and }
B=\Z_{p^{b(1)}}\times\cdots\times\Z_{p^{b(k)}}\]
be abelian $p$-groups in canonical form. Then $B$ is embeddable in $A$ if and
only if $b(j)\le a(j)$ for $j=1,\ldots,k$.
\end{lemma}

\begin{proof}
Suppose that the inequalities hold. Then for each $j$ we can choose a subgroup
$\Z_{p^{b(j)}}$ of $\Z_{p^{a(j)}}$; the direct product of these subgroups is
isomorphic to $B$.

Conversely, suppose that $B$ is embeddable in $A$. Then $B$ contains a 
subgroup $(\Z_{p^{b(j)}})^j$; in order to embed this in $A$, we require that
at least $j$ of $a(1),\ldots,a(k)$ are greater than or equal to $b(j)$ for
each $j$. Since the $a$ are non-increasing, this requires $a(j)\ge b(j)$.
\end{proof}

Now, to complete the proof of Proposition~\ref{p:minimum}, note that
in the notation before the proposition, $P$ embeds
all the $F_i$ if and only if $c(j)\ge a(i,j)$ for all $i$. So
$P$ is the smallest abelian $\F$-cover.\qed

\medskip

We can use this result to find the smallest abelian group containing every
abelian group of order $p^n$.

Define a function $f$ by the rule
\[f(n)=\sum_{k=1}^n\lfloor n/k\rfloor.\]

\begin{corollary} \label{c:abelian}
Let $\F$ be the set of all abelian groups of order $p^n$. There is a unique
smallest abelian $\F$-cover; its order is $p^{f(n)}$. If the answer to
Question~\ref{q:apgp} is affirmative, it is a minimum $\F$-cover.
\end{corollary}

\begin{proof}
In the notation introduced before Proposition~\ref{p:minimum}, we have
$c(k)=\lfloor n/k\rfloor$ for $k=1,\ldots,n$. For if the factors in the
canonical decomposition of an abelian group of order $p^n$ have orders
$p^{a(1)}$, $p^{a(2)}$, \dots, then 
\[ka(k)\le a(1)+\cdots+a(k)\le n,\]
so $a(k)\le\lfloor n/k\rfloor$; but there is a group of order $p^n$ with
$k$ invariant factors which are of nearly equal orders (that is, orders
$p^{\lfloor n/k\rfloor}$ or $p^{\lceil n/k\rceil}$); the $k$th of these in
non-increasing order has order $p^{\lfloor n/k\rfloor}$.

So the smallest abelian group covering $\F$ has order
$p^{\sum\lfloor n/k\rfloor}=p^{f(n)}$, as required.
\end{proof}

For $p^n=2^2$ and $2^3$, this gives respectively $8$ and $32$ for the
smallest abelian group containing all abelian groups of order $p^n$.
We have seen that these are also the orders of minimal covers for all
groups of these orders. Furthermore, no smaller group can cover all abelian
groups of these orders, by the proof of Theorem~\ref{thm:2covers}. 
So Question~\ref{q:apgp} has an affirmative answer in
these cases.

We note that the order of the smallest abelian cover of the class of abelian
groups of order $p^n$  is roughly $p^{n\log n}$, which can be contrasted with
the lower bound of $p^{cn^2}$ for a group covering every group of order $p^n$.
More precisely, $f(n)=n(\log n + 2\gamma - 1) + O(\sqrt{n})$, where $\gamma$ is
the Euler--Mascheroni constant (Dirichlet~\cite{pgld}).

The sequence of values of the function $f$ is sequence A006218 in the
On-Line Encyclopedia of Integer Sequences~\cite{oeis}. This gives many
interpretations of the sequence, but the one given here appears to be new.

Combining the result for the prime factors of an integer $n$, we obain the
following result.

\begin{theorem}
Let $n=p_1^{m_1}\cdots p_r^{a_r}$, where $p_1,\ldots,p_r$ are distinct primes.
Then the order of the smallest abelian group which embeds all abelian groups
of order $n$ is
\[A(n)=p_1^{f(m_1)}\cdots p_r^{f(m_r)},\]
where $f(m)=\sum_{k=0}^m\lfloor m/k\rfloor$.
\end{theorem}

Note that $A(n)=n$ if $n$ is squarefree, while $A(n)$ is roughly
$n^{\log\log n}$ if $n$ is a power of $2$. The values of $A(n)$ are not those
given by sequence A102361 in the OEIS~\cite{oeis}, although the first fifteen
terms agree (and these sequences agree at all fourth-power-free integers).

\medskip

We could ask whether similar results exist for soluble groups. But there is an
easy example to show that a set of soluble groups may have no soluble minimum
cover:

\begin{example} \label{ex:a5}
Let $A$ be the alternating group $A_4$ and $B$ the dihedral group $D_{10}$ of
order $10$.  Then $\lcm(|A|,|B|)=60$, and both groups are embeddable in $A_5$,
so $A_5$ is a minimum cover. There is no other cover of order $60$. For such a
group $G$ would act on the five cosets of $A$; it is easily seen that either
the action is faithful (whence $G\cong A_5$) or $A$ lies in the kernel (in
which case $G\cong A_4\times\Z_5$); but in the second case $G$ does not embed
the dihedral group.
\end{example}

\section{Further directions} \label{s:further}

We conclude by highlighting some further directions to be pursued, in addition to the various questions already posed in the paper.

First, the main definitions of this paper can be naturally dualised. Let $\F$ be a set of finite groups. A group $G$ is a \emph{dual $\F$-cover} if every group in $\F$ is isomorphic to a quotient of $G$. We say that a dual $\F$-cover $G$ is
\begin{itemize}
\setlength\itemsep{0pt}
\item \emph{minimal} if no proper quotient of $G$ is a dual $\F$-cover;
\item \emph{co-minimal} if no proper subgroup of $G$ is a dual $\F$-cover;
\item \emph{strongly minimal} if it is both minimal and co-minimal;
\item \emph{minimum} if no dual $\F$-cover has smaller order.
\end{itemize}
Note that a minimum dual cover is strongly minimal.

We have not investigated this concept except to note that, in the class of abelian groups, subgroups and quotients coincide (because of duality for abelian groups), so, for example, an abelian group is a minimum dual cover of a class of abelian groups if and only if it is a minimum cover.

\begin{question}
Investigate dual covers along the lines we have followed for covers.
\end{question}

Next, the work in Section~\ref{s:can} leads to the following question.

\begin{question}
For which classes $\X$ of groups, closed under the taking of subgroups and direct product, is it true that, if $\F$ is a finite set of $\X$-groups, then there is a minimum $\F$-cover which is an $\X$-group? 
\end{question}

As we have seen, this is true for cyclic groups and for nilpotent groups, but it is false for soluble groups (see Example~\ref{ex:a5}). We have been unable to resolve this question for abelian groups.

Finally, we give the answer to a question asked in an earlier version of the paper. The question asked: For which groups $G$ is it the case that $G$ is not a minimal cover of the set of its proper subgroups (equivalently, the set of its maximal subgroups)?

\begin{theorem}
If $G$ is not a minimal cover of the set of its proper subgroups, then $G$ is a $p$-group for some prime $p$, and all its maximal subgroups are isomorphic.
\end{theorem}

\begin{proof}
If $G$ is not of prime power order, then such a cover must contain all the Sylow subgroups of $G$, and so must be at least as large as $G$. If $G$ is a $p$-group, then all maximal subgroups have index $p$; so, if two are nonisomorphic, then again a cover must be at least as large as $G$.
\end{proof}

\begin{remark}
The $2$-groups with the property of the theorem were determined by \'Cepuli\'c~\cite{cep}. 
\end{remark}

\end{document}